\documentclass[11pt,a4paper]{article}
\usepackage{epsfig,amsfonts,amsgen,amsmath,amstext,amsbsy,amsopn,amsthm}
\usepackage{color}
\newtheorem{theorem}{Theorem}

\newtheorem{lemma}{Lemma}

\theoremstyle{definition}

\newtheorem{claim}{Claim}

\newtheorem{remark}[claim]{Remark}

\newtheorem{problem}{Problem}

\begin{document}

\title{\bf\Large Fan-type degree condition restricted to triples of
induced subgraphs ensuring Hamiltonicity}

\date{}

\author{Bo Ning\thanks{E-mail address: ningbo\_math84@mail.nwpu.edu.cn.}\\
\small Department of Applied Mathematics, School of Science,\\
\small Northwestern Polytechnical University,
\small Xi'an, Shaanxi 710072, P.R.~China\\}

\maketitle

\begin{abstract}
In 1984, Fan gave a sufficient condition involving maximum degree of
every pair of vertices at distance two for a graph to be
Hamiltonian. Motivated by Fan's result, we say that an induced
subgraph $H$ of a graph $G$ is $f$-heavy if for every pair of
vertices $u,v\in V(H)$, $d_{H}(u,v)=2$ implies that
$\max\{d(u),d(v)\}\geq n/2$. For a given graph $R$, $G$ is called
$R$-$f$-heavy if every induced subgraph of $G$ isomorphic to $R$ is
$f$-heavy. For a family $\mathcal{R}$ of graphs, $G$ is
$\mathcal{R}$-$f$-\emph{heavy} if $G$ is $R$-$f$-heavy for every
$R\in \mathcal{R}$. In this note we show that every 2-connected
graph $G$ has a Hamilton cycle if $G$ is
$\{K_{1,3},P_7,D\}$-$f$-heavy or $\{K_{1,3},P_7,H\}$-$f$-heavy,
where $D$ is the deer and $H$ is the hourglass. Our
result is a common generalization of previous theorems of Broersma
et al. and Fan on Hamiltonicity of 2-connected graphs.

\medskip
\noindent {\bf Keywords:} Combinatorial problem; Hamilton cycle; Fan-type degree condition; Induced subgraph; Claw
\smallskip

\noindent {\bf AMS Subject Classification (2000):} 05C38 05C45
\end{abstract}

\section{Introduction}
We use Bondy and Murty \cite{Bondy_Murty} for terminology and
notation not defined here and consider finite simple graphs only.

Let $G$ be a graph and $H$ be a subgraph of $H$. For two vertices $x,y\in
V(H)$, a shortest $(x,y)$-path in $H$ means that a path connecting $x$ and
$y$ with all vertices in $H$. The \emph{distance} between $x$ and $y$
in $H$, denoted by $d_{H}(x,y)$, is the length of a shortest $(x,y)$-path in $H$.
If $H=G$, we use $d(x,y)$ instead of $d_{G}(x,y)$.

A graph is called \emph{Hamiltonian} if it contains a Hamilton
cycle, i.e., a cycle passing through all its vertices. Checking
whether a given graph is Hamiltonian or not is a notorious
\emph{NP}-complete decision problem. Thus graphists drew their
attention to find sufficient conditions for the existence of
Hamilton cycles in graphs. The following sufficient condition for
the existence of Hamilton cycles in 2-connected graphs is well
known.

\begin{theorem}[Fan \cite{Fan}]\label{th1}
Let $G$ be a 2-connected graph on $n$ vertices. If
$\max\{d(x),d(y)\}\geq n/2$ for every pair of vertices $x$ and $y$
with $d(x,y)=2$, then $G$ is Hamiltonian.
\end{theorem}

There is another kind of sufficient conditions for Hamiltonicity of
graphs, called forbidden subgraph condition. Before we state some of
these results, we first introduce some terminology and notation.

Let $G$ be a graph and $H$ be a subgraph of $G$. If $H$ contains
every edge $xy\in E(G)$ with $x,y\in V(H)$, then $H$ is called an
\emph{induced subgraph} of $G$. For a given graph $R$, $G$ is
$R$-\emph{free} if $G$ contains no induced subgraph isomorphic to
$R$. For a family $\mathcal{R}$ of graphs, $G$ is
$\mathcal{R}$-\emph{free} if $G$ is $R$-free for every $R\in
\mathcal{R}$. The graph $K_{1,3}$ is called a \emph{claw}. Its only vertex with
degree 3 is called the \emph{center}, and other vertices are the
\emph{end vertices} of the claw. Throughout this note, instead of
$K_{1,3}$-free, we use the more common term \emph{claw-free}.

The following are two results on forbidden subgraph conditions for
Hamiltonicity of graphs.

\begin{theorem}[Broersma and Veldman \cite{Broersma_Veldman}]\label{th2}
Let $G$ be a 2-connected graph. If $G$ is claw-free and
$\{P_7,D\}$-free, then $G$ is Hamiltonian. (see Fig. 1)
\end{theorem}

\begin{theorem}[Faudree, Ryj\'{a}\v{c}ek and Schiermeyer \cite{Faudree_Ryjacek_Schiermeyer}]\label{th3}
Let $G$ be a 2-connected graph. If $G$ is claw-free and
$\{P_7,H\}$-free, then $G$ is Hamiltonian. (see Fig. 1)
\end{theorem}

\begin{center}
\begin{picture}(160,120)
\thicklines

\put(0,0){ \put(45,30){\circle*{6}} \put(45,30){\line(-1,1){25}}
\put(45,30){\line(1,1){25}} \put(20,55){\line(1,0){50}}
\multiput(20,55)(50,0){2}{\multiput(0,0)(0,30){3}{\put(0,0){\circle*{6}}}
\put(0,0){\line(0,1){60}}} \put(25,10){$D$ (Deer)}}

\put(90,0){\multiput(20,35)(0,75){2}{\multiput(0,0)(50,0){2}{\put(0,0){\circle*{6}}}
\put(0,0){\line(1,0){50}}} \put(45,72.5){\circle*{6}}
\put(20,35){\line(2,3){50}} \put(70,35){\line(-2,3){50}}
\put(10,10){$H$ (Hourglass)}}

\end{picture}

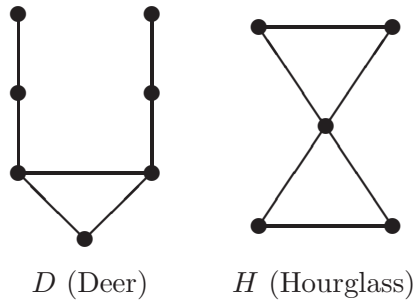
\small Fig. 1. Graphs $D$ and $H$.
\end{center}

Let $G$ be a graph on $n$ vertices. A vertex $v$ of $G$ is called
\emph{heavy} if $d(v)\geq n/2$.  Following
\cite{Broersma_Ryjacek_Schiermeyer}, an induced claw of $G$ is
called \emph{2-heavy} if at least two of its end vertices are heavy.
The graph $G$ is 2-\emph{heavy} if all induced claw of $G$ are
2-heavy. Thus 2-heavy graphs can be seen as graphs by restricting
Fan's condition to every induced claw.

Broersma et al. \cite{Broersma_Ryjacek_Schiermeyer} extended
Theorems \ref{th2} and \ref{th3} to a larger class of 2-heavy
graphs.

\begin{theorem}[Broersma, Ryj\'{a}\v{c}ek and Schiermeyer \cite{Broersma_Ryjacek_Schiermeyer}]\label{th4}
Let $G$ be a 2-connected graph. If $G$ is 2-heavy, and moreover,
$\{P_7,D\}$-free or $\{P_7,H\}$-free, then $G$ is Hamiltonian.
\end{theorem}

Let $G$ be a graph and $H$ be an induced subgraph. We say that $H$
is $f$-\emph{heavy} if for every pair of vertices $u,v\in V(H)$,
$d_{H}(u,v)=2$ implies that $max\{d(u),d(v)\}\geq n/2$. For a given
graph $R$, $G$ is called $R$-$f$-\emph{heavy} if every induced
subgraph of $G$ isomorphic to $R$ is $f$-heavy. For a family
$\mathcal{R}$ of graphs, $G$ is $\mathcal{R}$-$f$-\emph{heavy} if
$G$ is $R$-$f$-heavy for every $R\in \mathcal{R}$. Note that every
$R$-free graph is also $R$-$f$-heavy, and that a graph is 2-heavy is
equivalent to that it is claw-$f$-heavy.

By relaxing forbidden subgraph conditions to conditions in which the
subgraphs are allowed, but where Fan-type degree condition is
imposed on these subgraphs if they appear, we extend Theorem
\ref{th4} as follows.

\begin{theorem}\label{th5}
Let $G$ be a 2-connected graph. If $G$ is
$\{K_{1,3},P_7,D\}$-$f$-heavy or $\{K_{1,3},P_7,H\}$-$f$-heavy, then
$G$ is Hamiltonian.
\end{theorem}

\begin{remark}
It is easily seen that every graph satisfying the condition of
Theorem \ref{th1} or \ref{th4} also satisfies the one of Theorem
\ref{th5}. Furthermore, The graph provided in the following is a
Hamiltonian graph satisfying the condition of Theorem \ref{th5}, but
not the one of Theorem \ref{th1} or \ref{th4}.

Let $n\geq 16$ be an even integer and $K_{n/2}+K_{n/2-7}$ denotes
the union of two complete graphs $K_{n/2}$ and $K_{n/2-7}$. We
construct the graph $G$ with $V(G)=V(K_{n/2}+K_{n/2-7})\cup
\{x,y,z,u,v,w,t\}$ and $E(G)=E(K_{n/2}+K_{n/2-7})\cup
\{xy,xz,yz,yw,wu,zt,tv\}\cup \{xx',yx',zx': x'\in V(K_{n/2})\}\cup
\{uy',vy': y'\in V(K_{n/2-7})\}$.

This fact shows that Theorem \ref{th5} indeed strength Theorems
\ref{th1} and \ref{th4}.
\end{remark}

In the next section, we will give the proof of Theorem \ref{th5}. Some concluding remarks will be given in Section 3.

\section{Proof of Theorem 5}
Before giving the proof of Theorem \ref{th5}, we introduce some additional
terminology, and will list two useful lemmas.

Let $G$ be a graph and $C$ be a cycle of $G$. We denote by
$\overrightarrow{C}$ the cycle with a given orientation, and by
$\overleftarrow{C}$ the same subgraph with the reverse orientation.
For two vertices $x,y\in V(C)$, $\overrightarrow{C}[x,y]$ is denoted
by the consecutive vertices from $x$ to $y$ in $C$ by the direction
specified by $\overrightarrow{C}$, and $\overleftarrow{C}[y,x]$ is
the same vertices with the reverse order. For a vertex $x\in V(C)$,
$x^{+}$ denotes the \emph{successor} of $x$ on $\overrightarrow{C}$,
and $x^{-}$ denotes its \emph{predecessor}. Similarly, for a path
$P$ and $x,y\in V(P)$, $P[x,y]$ denotes the subpath of $P$ from $x$
to $y$.

Let $G$ be a graph on $n$ vertices.  Recall that a vertex of a graph
$G$ is \emph{heavy} if its degree is at least $n/2$. Otherwise, it
is \emph{light}. A cycle $C$ of $G$ is called a \emph{heavy cycle}
if it contains all the heavy vertices of $G$.

\begin{lemma}[Bollob\'{a}s and Brightwell \cite{Bollobas_Brightwell}, Shi \cite{Shi}]\label{le1}
Let $G$ be a 2-connected graph. Then $G$ contains a heavy cycle.
\end{lemma}

Next we introduce a new concept proposed in
\cite{Li_Ryjacek_Wang_Zhang} recently. In fact, it is a refinement of
the closure theory of Bondy-Chv\'atal. For the sake of convenience,
we rewrite it here. We use $\widetilde{E}(G)$ to
denote the set $\{xy:xy\in E(G)~or~d(x)+d(y)\geq n, x,y\in V(G)\}$.
Let $k\geq 3$ be an integer. A sequence of vertices $C=v_1v_2\ldots
v_kv_1$ is called an \emph{Ore-cycle} or briefly, \emph{o-cycle} of
$G$, if for every $i\in \{1,\cdots,k\}$, there holds $v_iv_{i+1}\in
\widetilde{E}(G)$, where the indices are taken modulo $k$.

\begin{lemma}[Li, Ryj\'{a}\v{c}ek, Wang and Zhang \cite{Li_Ryjacek_Wang_Zhang}]\label{le2}
Let $G$ be a graph and $C$ be an $o$-cycle of $G$. Then there exists a cycle $C'$ of $G$ such that $V(C)\subseteq V(C')$.
\end{lemma}

\noindent{}
{\bf {Proof of Theorem \ref{th5}}}

By Lemma \ref{le1}, $G$ contains a heavy cycle. Let $C$ be a longest
heavy cycle of $G$, fixed an orientation. Suppose that $G$ is not
Hamiltonian. Since $G$ is 2-connected, there is a path of length at
least 2, internally-disjoint with $C$, that connects two vertices of
$C$. Let $P=w_0w_1\ldots w_rw_{r+1}$ be such a path with $r$ as
small as possible, where $w_0=u\in V(C)$ and $w_{r+1}=v\in V(C)$.

\setcounter{claim}{0}
\begin{claim}\label{cl1}
Let $x\in V(P)\backslash \{u,v\}$ and $y\in \{u^-,u^+,v^-,v^+\}$.
Then $xy\notin \widetilde{E}(G)$.
\end{claim}

\begin{proof}
Without loss of generality, assume that $y=u^-$. Suppose that $xy\in
\widetilde{E}(G)$. Then $C'=yxP[x,u]\overrightarrow{C}[u,y]$ is an
$o$-cycle containing all the vertices of $C$ and longer than $C$. By
Lemma \ref{le2}, there is a longer cycle containing all the vertices
in $C$, that is, a longer heavy cycle in $G$, a contradiction. The
other assertions can be proved similarly.
\end{proof}

\begin{claim}\label{cl2}
$u^-u^+\in \widetilde{E}(G)$, $v^-v^+\in \widetilde{E}(G)$.
\end{claim}
\begin{proof}
Suppose that $u^-u^+\notin E(G)$.  By Claim \ref{cl1}, $\{u,u^-,u^+,w_1\}$
induces a claw. By the choice of $C$, $w_1$ is light. Since $G$ is 2-heavy,
we have $d(u^-)+d(u^+)\geq n$. This implies that $u^-u^+\in \widetilde{E}(G)$.
Similarly, we can prove the other assertion.
\end{proof}

\begin{claim}\label{cl3}
$uv^{\pm}\notin \widetilde{E}(G)$, $vu^{\pm}\notin \widetilde{E}(G)$,
$u^-v^-\notin \widetilde{E}(G)$, $u^+v^+\notin \widetilde{E}(G)$.
\end{claim}

\begin{proof}
Suppose that $uv^-\in \widetilde{E}(G)$. By Claim \ref{cl2},
$u^+u^-\in \widetilde{E}(G)$. Then
$C'=uv^-\overleftarrow{C}[v^-,u^+]u^+\break
u^-\overleftarrow{C}[u^-,v]P[v,u]$ is an $o$-cycle containing all
vertices in $C$ and longer than $C$, a contradiction. Suppose
$u^-v^-\in \widetilde{E}(G)$. Then
$C'=u^-v^-\overleftarrow{C}[v^-,u]P[u,v]\overrightarrow{C}[v,u^-]$
is an $o$-cycle containing all vertices in $C$ and longer than $C$,
a contradiction. The other assertions can be proved similarly.
\end{proof}

Let $y_1$ be the first vertex on $\overrightarrow{C}[u,v]$ such that
$uy_1\notin E(G)$, $y_2$ be the first vertex on
$\overrightarrow{C}[v,u]$ such that $vy_2\notin E(G)$. By Claim
\ref{cl3}, $uv^-\notin E(G)$ and $vu^-\notin E(G)$. Thus, $y_1$ and
$y_2$ are well-defined.
\begin{claim}\label{cl4}
Let $w\in \{w_1,\cdots,w_r\}, x\in \overrightarrow{C}[u^+,y_1]$
and $y\in \overrightarrow{C}[v^+,y_2]$. Then we have

(1) $wx\notin \widetilde{E}(G), wy\notin \widetilde{E}(G)$;

(2) $uy\notin \widetilde{E}(G), vx\notin \widetilde{E}(G)$;

(3) $xy\notin \widetilde{E}(G)$.
\end{claim}

\begin{proof}
(1) Suppose that $wx\in \widetilde{E}(G)$. By Claim \ref{cl1},
$x\neq u^+$ and this implies that $ux^-\in E(G)$. Then
$C'=P[u,w]wx\overrightarrow{C}[x,u^-]u^-u^+\overrightarrow{C}[u^+,x^-]x^-u$
is an $o$-cycle longer than $C$ and contains all vertices in $C$, a
contradiction. The other assertion can be proved similarly.

(2) Suppose that $uy\in \widetilde{E}(G)$. By Claim \ref{cl3},
$y\neq v^+$ and this implies that $vy^-\in E(G)$. Then
$C'=uy\overrightarrow{C}[y,u^-]u^-u^+\overrightarrow{C}[u^+,v^-]v^-v^+\overrightarrow{C}[v^+,y^-]y^-vP[v,u]$
is an $o$-cycle longer than $C$ and contains all vertices in $C$, a
contradiction. The other assertion can be proved.

(3) Suppose that $xy\in \widetilde{E}(G)$. By Claim \ref{cl2},
$u^-u^+\in \widetilde{E}(G)$ and $v^-v^+\in \widetilde{E}(G)$. Now
$C'=P[u,v]vy^-\overleftarrow{C}[y^-,v^+]v^+v^-\overleftarrow{C}[v^-,x]xy\overrightarrow{C}[y,u^-]u^-u^+\overrightarrow{C}[u^+,x^-]x^-u$
(if $x\neq u^+$ and $y\neq v^+$) or
$C'=P[u,v]vy^-\overleftarrow{C}[y^-,v^+]v^+v^-\overleftarrow{C}[v^-,u^+]u^+y\overrightarrow{C}[y,u]$
(if $x=u^+$ and $y\neq v^+$) or
$C'=P[u,v]\overleftarrow{C}[v,x]xv^+C[v^+,u^-]u^-u^+C[u^+,x^-]x^-u$
(if $x\neq u^+$ and $y=v^+$) is an $o$-cycle longer than $C$ and
contains all vertices in $C$, a contradiction.
\end{proof}

\begin{claim}\label{c5}
$u^-u^+\in E(G)$ or $v^-v^+\in E(G)$.
\end{claim}

\begin{proof}
Suppose that $u^-u^+\notin E(G)$ and $v^-v^+\notin E(G)$. By Claim
\ref{cl2}, we have $d(u^-)+d(u^+)\geq n$ and $d(v^-)+d(v^+)\geq n$.
Thus, we obtain $d(u^-)+d(v^-)\geq n$ or $d(u^+)+d(v^+)\geq n$,
contradicting Claim \ref{cl3}.
\end{proof}

By Claim \ref{c5}, without loss of generality, we assume that $u_{-1}u_1\in E(G)$.

\begin{claim}
$uv\in E(G)$.
\end{claim}

\begin{proof}
Suppose that $uv\notin E(G)$. By Claim \ref{cl4},
$\{y_1,y_1^-,u,w_1,\ldots,w_r,v,y_2^-,y_2\}$ induces a $P_{6+r}$,
where $r\geq 1$. Since $G$ is $P_7$-$f$-heavy, $G$ is also
$P_{6+r}$-$f$-heavy. By the choice of $C$, $w_1$ and $w_r$ are
light. It follows that $y_1^-$ and $y_2^-$ are heavy, and this
implies $y_1^-y_2^-\in \widetilde{E}(G)$, contradicting Claim
\ref{cl4} (3).
\end{proof}

\begin{claim}
$r=1$.
\end{claim}

\begin{proof}
Suppose $r\geq 2$. Since $r\geq 2$ and by the choice of the path $P$, we have $w_1v\notin E(G)$ and $w_ru\notin E(G)$. By Claim \ref{cl3},
we obtain $uv^+\notin E(G)$ and $u^-v\notin E(G)$. By Claim
\ref{cl1}, $w_1u^-\notin E(G)$ and $w_rv^+\notin E(G)$. Thus each of
$\{u,w_1,u^-,v\}$ and $\{v,w_r,v^-,u\}$ induces a claw. Since each
of $\{w_1,w_r\}$ is light and $G$ is 2-heavy, $u^-$ and $v^-$ are
heavy. Hence $u^-v^-\in \widetilde{E}(G)$, which contradicts Claim
\ref{cl3}.
\end{proof}

Note that $G$ is $D$-$f$-heavy or $H$-$f$-heavy. If $G$ is
$D$-$f$-heavy, then by Claim \ref{cl4}, $\{y_1,y_1^-,u,\break
w_1,v,y_2^-,y_2\}$ induces a $D$. Since $w_1$ is light, $y_1^-$ and
$y_2^-$ are heavy. It follows that $y_1^-y_2^-\in \widetilde{E}(G)$,
which contradicts Claim \ref{cl4} (3). Now, we assume that $G$ is
$H$-$f$-heavy. By Claims \ref{cl1} and \ref{cl3},
$\{u^-,u,u^+,w_1,v\}$ induces an $H$. It follows that $u^-$ is
heavy. If $v^-v^+\in E(G)$, then by Claims \ref{cl1} and \ref{cl3},
$\{v^-,v,v^+,w_1,u\}$ induces an $H$. Similarly, we have $v^-$ is
heavy. If $v^-v^+\notin E(G)$, then $\{v^-,v,v^+,w_1\}$ induces a
claw. Since $w_1$ is light and $G$ is 2-heavy, we have $v^-$ is
heavy. In these two cases, we obtain $u^-v^-\in \widetilde{E}(G)$,
which contradicts Claim \ref{cl3}.

The proof is complete. \hfill $\Box$

\section{Concluding remarks}
In this note, we give a new sufficient condition for Hamiltonicity
of graphs by restricting Fan's condition to triples of induced
subgraphs of graphs.

In fact, the idea that one can guarantee Hamiltonicity of graph by
restricting Fan's condition to pairs of induced subgraphs dated from
Bedrossian, Chen and Schelp \cite{Bedrossian_Chen_Schelp}. Later,
Chen, Wei and Zhang \cite{Chen_Wei_Zhang_0,Chen_Wei_Zhang_5}, and
Li, Wei and Gao \cite{Li_Wei_Gao} got related results with this
similar idea. Note that Bedrossian \cite{Bedrossian} characterized
all pairs of forbidden subgraphs $\{R,S\}$ for Hamiltonicity of
2-connected graphs. Thus we can pose this problem: which two
connected graphs $R$ and $S$ other than $P_3$ imply that every
2-connected $\{R,S\}$-\emph{f}-heavy graph is Hamiltonian? Recently,
this problem has been completely solved in \cite{Ning_Zhang}.

Brousek \cite{Brousek} gave a complete characterization of triples
of connected graphs $\{K_{1,3},R,S\}$ such that a graph $G$ being
2-connected and $\{K_{1,3},R,S\}$-free is Hamiltonian. Thus we can
pose the following problem naturally.

\begin{problem}
To characterize all possible triples of connected graphs
$\{K_{1,3},R,S\}$ such that every 2-connected graph $G$ being
$\{K_{1,3},R,S\}$-$f$-heavy is Hamiltonian.
\end{problem}

\section*{Acknowledgement}
This work is supported by NSFC (No.~11271300) and the Doctorate
Foundation of Northwestern Polytechnical University (cx201326). The
author would like to express gratitude to the editors and reviewers,
whose invaluable suggestions have improved the presentation of this
work.

\end{document}